\documentclass[a4paper,reqno]{amsart}

\usepackage{graphicx,pdfsync, amsmath, amssymb, latexsym,amsfonts,amsbsy,color,subcaption,hyperref}
\hypersetup{
	colorlinks=true,
	linkcolor=blue,
	filecolor=blue,      
	urlcolor=blue,
	citecolor=blue,
}

\usepackage[parfill]{parskip}
\usepackage{tikz,pgfplots,caption}
\usetikzlibrary{patterns}
\usepackage[top=1.1in, bottom=1in, left=1in, right=1in]{geometry}

\newtheorem{theorem}{Theorem}[section]
\newtheorem{lemma}[theorem]{Lemma}
\newtheorem{proposition}[theorem]{Proposition}

\newtheorem{corollary}[theorem]{Corollary}
\newtheorem{remark}[theorem]{Remark}

\usepackage{times}
\usepackage{setspace}
\def \RR {\mathbb{R}}

\def \l {\lambda}

\newcommand{\comment}[1]{}

\DeclareMathOperator{\tr}{Tr}

\numberwithin{equation}{section}

\begin{document}

	\title{Energy equality for the Navier-Stokes equations in weak-in-time Onsager spaces}
	

	\author{Alexey Cheskidov}

	\author{Xiaoyutao Luo}
	\address{Department of Mathematics, Statistics and Computer Science,
		University of Illinois At Chicago, Chicago, Illinois 60607}
	%\curraddr{Department of Mathematics, Statistics and Computer Science,University of Illinois At Chicago, Chicago, Illinois 60607}
	\email{acheskid@uic.edu,xluo24@uic.edu}
	
	\thanks{The authors were partially supported by the NSF grant DMS--1517583}
	
	%    General info
	%\subjclass[2000]{76D03, 35Q35}
	
	\date{\today}
	\begin{abstract}
		Onsager's conjecture for the 3D Navier-Stokes equations concerns the validity of energy equality of weak solutions with regards to their smoothness. In this note we establish energy equality for weak solutions in a large class of function spaces. These conditions are weak-in-time with optimal space regularity and therefore weaker than all previous classical results. Heuristics using intermittency argument suggests the possible sharpness of our results.
	\end{abstract}
	
	\maketitle
	\section{Introduction}
	We consider the three-dimensional incompressible Navier-Stokes equations (3D NSE)
	\begin{align}\label{eq:3dNSE}
	& \partial_t u+ (u\cdot \nabla )u -\nu \Delta u = -\nabla p, \nonumber \\
	& \nabla \cdot u =0 , \\
	& u(x,0) = u_0(x), \nonumber
	\end{align}
	where $u(x,t)$ is the unknown velocity, $p(x,t)$ is the scalar pressure, $\nu >0$ is the kinematic viscosity. We also restrict attention to spatial domain $\Omega = \RR^3$ or $ \mathbb{T}^3$.
	
	It is known from classical results by Leray \cite{L34} that for divergence-free initial data $u_0 \in L^2$, there exists a weak solution to \eqref{eq:3dNSE} up to a specified time $T < \infty $ so that the following energy inequality holds
	\begin{equation}\label{eq:energy_ineq}
	\|u(t) \|_2^2 + 2\nu \int_{t_0}^{t} \|\nabla u \|_2^2 \leq \|u(t_0)\|_2^2,
	\end{equation}
	for all $t \in [0,T]$ and a.e $t_0 \in [0,t]$ including $0$. Weak solutions satisfying \eqref{eq:energy_ineq} are called Leray-Hopf weak solutions. These solutions have additional regularity $L_t^2 H_x^1$ and other analytic properties. On the other hand regular solutions to the 3D NSE satisfy the corresponding energy equality:
	\begin{equation}\label{eq:energy_eq}
	\|u(t) \|_2^2 + 2\nu \int_{t_0}^{t} \|\nabla u \|_2^2 = \|u(t_0)\|_2^2.
	\end{equation}
	
	It remains an open question whether energy equality is valid for Leray-Hopf weak solutions or general weak solutions. The difference between \eqref{eq:energy_eq} and \eqref{eq:energy_ineq} is the presence of anomalous energy dissipation due to nonlinearity. This phenomenon is generally restricted to rough solutions of fluid equations \cite{E} and the famous Onsager conjecture for the 3D Euler and the 3D NSE seeks to find the threshold regularity for energy conservation. More specifically, it asserts that any weak solution that is more regular than the threshold conserves energy, and there exists a weak solution below the regularity threshold that experiences anomalous dissipation.
	
	For the 3D Euler the conjecture is basically settled. On the one hand, in \cite{CET} the energy conservation was established for weak solutions in $L^3 B^{\alpha}_{3,\infty} $, $\alpha >\frac{1}{3}$. Later on this condition has been weakened to $L^3 B^{3}_{3,c_0} $ in \cite{ccfs}, which so far is the weakest condition for the energy conservation of the 3D Euler equations. On the other hand, the existence of anomalous dissipative weak solutions with less than Onsager regularity has been proven by several authors \cite{Buc,DS2,DS3,DS4}. These methods are based on convex integration originated from work of Nash on isometry problem in differential geometry \cite{Nash}. Finally in \cite{I} Isett constructed, for any $\alpha < \frac{1}{3}$, weak solutions in $C_tC_x^\alpha$ that fail to conserve energy, closing the conjecture from the other direction.
	
 	For the 3D NSE there is still a lot left to be done. Let us review the progress so far. The energy equality was proven by Lions for weak solutions in $L^4_t L^4_x$ \cite{Lio}. Shinbrot \cite{Shin} improved upon this result, showing the energy equality under the condition $u \in L^q_t L^p_x $ such that $ \frac{2}{p} + \frac{2}{q} \leq 1$ with $ p \geq 4$. It worth noting that the result of \cite{ccfs} also implies energy equality for $u \in L^3 B^{3}_{3,\infty}$ and many previous results including \cite{Lio} and \cite{Shin} can be recovered using interpolation methods (see Section \ref{section2-2}). Recent work by Leslie and Shvydkoy \cite{LS1} established the energy equality under new $ L^q_t L^p_x $ conditions using local energy estimates.  While for the 3D Euler equations the conjecture is basically settled for both directions, the existence of a weak solution to the 3D NSE that exhibits anomalous dissipation was not known until very recently. In \cite{BV1}, the authors show the nonuniqueness and anomalous dissipation of weak solutions for the 3D NSE in $C_tH^{\beta}_x $ for some small $\beta >0$ using the technique developed from settling Onsager's conjecture for the 3D Euler equations. This motivates for a thorough investigation of the energy equality for the 3D NSE. There is still a possibility that the linear term in the 3D NSE prohibits some anomalous energy dissipation scenarios.

	\subsection{Main results}
	
	Recall that a weak solutions $u(t)$ of the 3D NSE is a weakly continuous $L^2$ valued function in the class $u\in L^2_tH^1_x$ satisfying \eqref{eq:3dNSE} in the sense of distribution.
	
	\begin{theorem}\label{theorem:weak_lbg}
		Suppose $1 \leq \beta<p \leq \infty$ are such  that $\frac{2}{p} + \frac{1}{\beta}<1$. If a weak solution $u(t)$ of the 3D NSE satisfies
		\begin{equation}\label{eq:weak_lbg}
		u \in L^{\beta,w}(0,T;B^{\frac{2}{\beta}+\frac{2}{p}-1}_{p,\infty} ),
		\end{equation}
		then $u(t)$ satisfies the energy equality on $[0,T]$.
	\end{theorem}
	
	In view of $L_t^q L_x^p$ conditions for energy equality, \eqref{eq:type1} is weaker than the result of Shinbrot \cite{Shin}. It is worth noting that we have not improved over the endpoint Onsager space $L^3 B^{\frac{1}{3}}_{3,\infty}$. Nonetheless, the importance of the condition \eqref{eq:weak_lbg} is that it is weak-in-time with the optimal Onsager spacial regularity exponent $\frac{2}{\beta}+\frac{2}{p}-1$. In contrast, Ladyzhenskaya-Prodi-Serrin regularity conditions \cite{KL,P,S} require $\frac{2}{\beta}+\frac{3}{p}-1$ as spacial regularity in order to rule out possible blowups. 
	
	We also obtain the energy equality for various Type-I blowups as a useful corollary:
	
	\begin{corollary}\label{corollary:type-1}
		If a strong solution $u(t)$ of the 3D NSE on $[0,T)$ satisfies
		\begin{equation}\label{eq:type1}
		\| u(t) \|_{B^0_{p,\infty}} \lesssim \frac{1}{(T-t)^{\frac{1}{2}-\frac{1}{p}}}, \qquad 0< t <T,
		\end{equation}
		for some $p > 4$,
		then $u(t)$ does not lose energy at time $T$.	
	\end{corollary}
	We note that the classical Type-I blowup $\|u(t) \|_\infty \lesssim \frac{1}{ \sqrt{T-t}  }$, for which the loss of energy was ruled out by a recent result by Leslie and Shvydkoy \cite{LS1}, is covered by \eqref{eq:type1}. In addition, when $p < \infty$, condition \eqref{eq:type1} is weaker than the critical 3D NSE scaling reflected in the following upper bound if a blowup at time $T$ occurs:
	\begin{equation} \label{eq:Lpupperbound}
	\| u(t) \|_p \gtrsim \frac{1}{(T-t)^{\frac{1}{2}- \frac{3}{2p}}}, \qquad p > 3.
	\end{equation}
	Remarkably, the worst intermittency dimension for \eqref{eq:weak_lbg} and \eqref{eq:type1} is $d=1^{-}$ in contrast to
	regularity criteria, such as \eqref{eq:Lpupperbound}, where $d=0$ is the worst case scenario (more on this in Section \ref{section2}). 
	
	%This is due to the effect of intermittency (more on this in Section \ref{section2}). Heuristically, taking the intermittency dimension $d \in [0,3]$ into account, one should expect 
	%$$
	%\| u(t) \|_p \sim \frac{1}{(T-t)^{\frac{1}{2}- \frac{3-d}{2p}}}.
	%$$
	%Roughly speaking, one should expect $d=1^-$ when $\frac{2}{p} + \frac{1}{\beta} < 1$ and $d=0$ when $\frac{2}{p} + \frac{1}{\beta} > 1$.

	Our last result extends previous conditions in the regime $\frac{2}{p} + \frac{1}{\beta} \geq 1$. Unlike Theorem \ref{theorem:weak_lbg}, the scaling of this result corresponds to $d=0$ and thus the spacial regularity is different than that of Theorem \ref{theorem:weak_lbg}.
	
	\begin{theorem}\label{theorem:type2}
		Suppose $1 \leq p \leq \infty$, $0 < \beta \leq 3$ so that $\frac{2}{p} + \frac{1}{\beta} \geq 1$. If a weak solution $u(t)$ of the 3D NSE satisfies
		\begin{equation}\label{eq:type2}
		u \in  L^\beta\big(0,T;   B^{\frac{5}{2\beta} + \frac{3}{p} - \frac{3}{2} }_{p,\infty} \big) ,
		\end{equation}
		then $u(t)$ satisfies energy equality on $[0,T]$.
	\end{theorem}
	
	Note that only the regime $0 <\beta < 1$ in \eqref{eq:type2} is new since in this case $L^\beta $ is not a normed space and hence one can not use interpolation technique.

	The rest of the note is organized as follows. In Section \ref{section2} we give some heuristics using the intermittency dimension to show the sharpness of our result and summarize previous works on the conditions for the energy equality. Section \ref{section3} is devoted to preliminaries and tools we used, mainly the Littlewood-Paley theory and estimates involving the energy flux. Finally, we prove the main results in Section \ref{section4}.

	\section{Heuristics and comparison with previous results}\label{section2}
	\subsection{Heuristics}
	Consider the following scenario
	for the anomalous energy dissipation. Assume that at each time $t$, the total energy $E =\|u(t) \|_2^2$ is concentrated in a dyadic shell of radius $\l(t)$ in the Fourier space. If the energy is
	of order one, the time it takes for it to transfer to the shell of radius $2\l$ is
	\[
	T = \frac{\text{Energy}}{\text{Flux}}.
	\]
	Assuming the flow has intermittency dimension $d\in [0,3]$, namely
	$$
	\|u \|_2 \sim \lambda^{\frac{3-d}{2} } \| u\|_\infty,
	$$
	it follows that 
	$$
	\text{Flux} \sim  \l^{\frac{5-d}2} E^{3/2} .
	$$ 
	This implies that
	\[
	\l(T^*-T)\sim (TE^{\frac12})^{\frac{2}{d-5}},
	\]
	where $T^*$ is the time of blow-up.
	No note that the range $(1,3]$ for the intermittency dimension $d$ is eliminated 
	because the linear term dominates in that regime and hence the solution is regular
	and has to satisfy the energy equality. Heuristically, the linear and
	nonlinear terms scale as
	\[
	L =\text{Enstrophy}=  \l^2 E, \qquad N=\text{Flux} = \l^{\frac{5-d}{2}} E^\frac{3}{2}.
	\]
	Hence,
	\[
	L > N, \quad \text{provided } d>1.
	\]
	Moreover, one can actually exclude the case $d=1$ by noticing that
	the enstrophy behaves as
	\[
	\text{Enstrophy} = \|u(T^*-T)\|_{H^\alpha}^2 \sim \l^2E \sim T^{\frac{4}{d-5}},
	\]
	which is not integrable when $d\geq 1$. Here we assumed that the energy $E$ is of order one, i.e., some chunk of energy
	escapes to the infinite wavenumber. This suggests that the range for the intermittency dimension is $d \in [0,1)$.
	
	Now we can compute the speed with which various norms are allowed to blow up, for instance
	\begin{equation} \label{eq:typeIenergyHs}
	\|u(T^*-T)\|_{H^\alpha}  \sim \l^{ \alpha} \sqrt{E} \sim T^{\frac{2\alpha}{d-5}}.
	\end{equation}
	Optimizing this over $d\in[0,1)$ we obtain that the extreme intermittency $d=0$ is the the ``worst'' (which is usually the case),
	and the condition $u \in L^{\frac{5}{2s}}_t H^\alpha_x$
	should imply the energy equality. In particular, we can see the familiar scaling $u \in L^{3}_t H^{\frac{5}{6}}_x$ when $\alpha=5/6$.
	
	This heuristics becomes more surprising when we look at $L^p$-based spaces. In particular, for $L^\infty$-based spaces we have
	\[
	\|u(T^*-T)\|_{B^{\alpha}_{\infty,\infty}} \sim \l^{\alpha+\frac{3-d}{2}} \sqrt{E} \sim T^{\frac{2\alpha+3- d}{d-5}}.
	\]
	Optimizing this over $d\in[0,1)$  again, we obtain that the intermittency $d$ near $1$ is the the ``worst'', which is unusual.

	In general, for $L^p$-based spaces we have (interpolating between $L^2$ and $L^\infty$)
	\[
	\|u(T^*-T)\|_{B^{\alpha}_{p,\infty}} \sim \l^{ \alpha+(1-\frac{2}{p}) \frac{3-d}{2}  } E \sim T^{f(\alpha,p,d)},
	\]
	where
	\begin{equation} \label{eq:f}
	f(\alpha,p,d) =  \frac{2\alpha+(1-\frac{2}{p})(3-d)}{d-5}.
	\end{equation}
	To find the ``worst'' value of the intermittency dimension $d$, we must ask the following
	question: What is the smallest possible value of the $B^\alpha_{p, \infty}$-norm at time $T^*-T$ 
	so that
	the loss of energy can still occur at time $T^*$? Observe that $\frac{\partial}{\partial d} f$ has the
	same sign as $1-\frac{2}{p}-\alpha$. Therefore,
	\begin{align*}
	d &=0 \text{ is the worst intermittency dimension for } \alpha>1-\frac{2}{p}, \\
	d &= 1^- \text{ is the worst intermittency dimension for } \alpha<1-\frac{2}{p}. 
	\end{align*}
	Then the following optimal time integrability exponent $\beta=-\frac{1}{f}$ can be obtained from \eqref{eq:f}:
	\[
	\frac{1}{\beta}=
	\left\{
	\begin{split}
	& \frac{\alpha}{2}+\frac{1}{2}-\frac{1}{p}, \quad  &\text{when} \quad \alpha\leq1-\frac{2}{p},\\
	& \frac{2}{5}\alpha+\frac{3}{5}-\frac{6}{5p}, \quad  &\text{when} \quad \alpha>1-\frac{2}{p}.
	\end{split}
	\right.
	\]
	
	In what follows we often use $p$ and $\beta$ to parametrize different cases instead of using $\alpha$ and $p$. Simple algebra shows that 
	\begin{align*}
	\alpha < 1 - \frac{2}{p}   \Leftrightarrow \frac{2}{p} +  \frac{1}{\beta} <1 \quad \text{and} \quad \alpha \geq 1 - \frac{2}{p}    \Leftrightarrow \frac{2}{p} +  \frac{1}{\beta} \geq 1 .
	\end{align*}
	
	\subsection{Comparison with previous works}\label{section2-2}
	If a weak solution of the 3D NSE belongs to the Onsager space $L^3 B^{\frac{1}{3}}_{3,\infty}$, then it satisfies the energy equality. This follows from the estimate on the energy flux done in \cite{ccfs}
	and implies classical results on energy equality, such as \cite{Lio,Shin},  via interpolation with the energy class $L_t^\infty L_x^2 \cap L_t^2 H_x^1$. We provide a concise argument below.
	
	We will determine the range of parameters $1 \leq \beta \leq \infty  $, $ 1 \leq p\leq  \infty$, and $\alpha \in \mathbb{R}$ so that the following embeddings hold:
	$L^\beta B^{\alpha}_{p,\infty} \cap L^2 H^1 \cap L^\infty L^2 \subset L^\frac{1}{3} B^{\alpha'}_{p',q'} \subset  L^\frac{1}{3} B^{\frac{1}{3}}_{3,\infty}$. The goal is to find the minimal space regularity exponent $\alpha$ given $\beta$ and $p$. In view of H\"older interpolation in time and Besov interpolation in space, for $x$ and $y$ satisfying $0  \leq x+y \leq 1, \, 0\leq  x\leq 1 ,\, 0\leq  y \leq 1 $ we have the following relations:
	\begin{align}
	\frac{1}{3} &= x \cdot \frac{1}{\beta} + y \cdot  \frac{1}{2}+(1-x-y) \frac{1}{\infty}, \label{1} \\
	\frac{1}{p'} &= x \cdot \frac{1}{p} + y \cdot  \frac{1}{2}+(1-x-y) \frac{1}{2}, \label{2}\\
	\alpha' &=x \cdot \alpha + y \cdot  1, \label{3}\\
	\alpha' & \geq 3\left(\frac{1}{p'} -\frac{1}{3}\right) +\frac{1}{3}. \label{4}
	\end{align}
	After substitutions we find $\alpha \geq (\frac{3}{p} +\frac{2}{\beta} -\frac{3}{2} ) + \frac{1}{6x} $. So to find the minimal $\alpha$ we need to determine the range of $x$. 
	
	First, $p' \leq 3$ in order to make sure that the Besov embedding $ B^{\alpha'}_{p',q'} \subset   B^{\frac{1}{3}}_{3,\infty}$ holds, which is equivalent to $ \frac{1}{x} \geq 3- \frac{6}{p}$ due to \eqref{2}. Second, the inequality $ x+y \leq 1$ is equivalent to $ \frac{1}{x} \geq 3-\frac{6}{\beta}$, and $y \geq 0$ is equivalent 
	to $ \frac{1}{x} \geq  \frac{3}{\beta} $ thanks to \eqref{1}. 
	
	Therefore given time integrability $\beta$ and space integrability $p$, the minimal spacial regularity exponent that implies the energy equality is 
	\[ 
	\alpha =
	\begin{cases} 
	\frac{2}{\beta} + \frac{2}{p} -1 & \beta \geq 3 \, \text{and } \, p \geq \beta,  \\
	\frac{1}{\beta} + \frac{3}{p} -1 & \beta \geq 3 \, \text{and } \,  p \leq \beta,  \\
	\frac{5}{2\beta} + \frac{3}{p} -\frac{3}{2} & \beta \leq 3 \, \text{and } \, \frac{1}{\beta}+ \frac{2}{p}  \geq 1,  \\
	\frac{2}{\beta} + \frac{2}{p} -1 & \beta \leq 3 \, \text{and } \, \frac{1}{\beta}+ \frac{2}{p}  \leq 1.  \\
	\end{cases}
	\]
	
	We can summarize classical results as follows.
	\begin{figure*}[hb]
		\begin{tikzpicture}
		\draw [<->] (0,5) node[left]{$ \frac{1}{p}$}  --(0,0)  -- (5,0) node[below]{$ \frac{1}{\beta} $};
		\draw  (1.33, 0) -- node[below]{$ \frac{1}{3} $} (1.33, -0.1); 
		%\draw  (2, 0) -- node[below]{$ \frac{1}{2} $} (2, -0.1); 
		\draw  (4, 0) -- node[below]{$ 1 $} (4, -0.1); 
		\draw  (-0.1, 4) -- node[left]{$ 1 $} (0, 4);
		\draw  (-0.1, 2) -- node[left]{$ \frac{1}{2} $} (0, 2);
		\draw  (-0.1, 1.33) -- node[left]{$ \frac{1}{3} $} (0, 1.33); 
		\draw   (0,0) --(1.33,1.33)  ; \draw   (0,4) --(4,4) --(4,0) ;
		\draw [dashed] (0,2)  -- (1.33,1.33) ; \draw   (1.33,1.33)  -- (4,0) ;   
		\draw [dashed] (1.33,0)  -- (1.33,1.33) ; \draw   (1.33,1.33)  -- (1.33,4) ;   
		\draw [->] (4.5,1.5)node[right]{$\frac{2}{p}+\frac{1}{\beta} = 1$} -- (2.8,0.7)  ;
		\draw [dashed] (0,1.33) -- (2,1.33) ;
		\fill[pattern=north east lines,opacity=0.5] (1.33,1.33)-- (0,0) --(4,0)-- (1.33,1.33) ;
		\fill[opacity=0.2] (0,4)-- (0,0) -- (1.33,1.33) --(1.33,4);
		\draw[->]    (2,-0.5)node[right]{ $\alpha = \frac{2}{\beta}+ \frac{2}{p}- 1$ } --(1.5,0.5) ;
		\draw[->]   (1.2,4.2) node[above]{$\alpha = \frac{1}{\beta}+ \frac{3}{p}- 1$} -- (0.6,2.8) ;
		\draw    (1.3,2.5)node[right]{$\alpha = \frac{5}{2\beta}+ \frac{3}{p}- \frac{3}{2}$};
		\end{tikzpicture}
		\caption{Regions for Lemma \ref{lemma:classical}}
		\label{fig:classical1}
	\end{figure*}
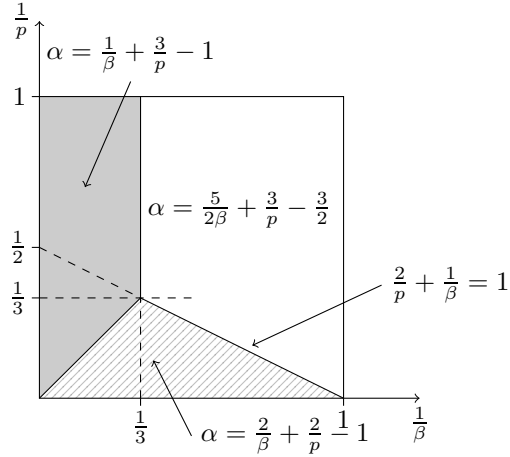
	
	\begin{lemma}[Classical results]\label{lemma:classical}
		If $u(t)$ is a weak solution to the 3D NSE satisfying either one of following three conditions
		\begin{align}
		u &\in L^\beta B^{\frac{2}{\beta} + \frac{2}{p} -1 }_{p,\infty} \quad \text{for some} \quad \frac{1}{\beta}+ \frac{2}{p} \leq 1 \text{ and }  p\geq \beta , \label{eq:previous_energy1} \\
		u &\in L^\beta B^{\frac{5}{2\beta}+ \frac{3}{p}-\frac{3}{2}}_{p,\infty}  \quad \text{for some} \quad  \frac{1}{\beta}+ \frac{2}{p} \geq 1, \ 1 \leq \beta \leq 3, \text{ and } p\geq1, \label{eq:previous_energy2} \\
		u &\in L^\beta B^{\frac{1}{ \beta}+ \frac{3}{p}-1 }_{p,\infty}  \quad \text{for some}\quad  \beta \geq 3 \text{ and }  1\leq p   \leq \beta \label{eq:previous_energy3}
		\end{align}
		then $u$ satisfies energy equality.
	\end{lemma}
	
	\begin{remark}
		Taking $\alpha =0$ and $\beta=p=4$ we can see the $L_t^4 L_x^4$ result by Lions \cite{Lio}. Moreover, if $u \in L_t^q L_x^p$ with $\frac{2}{q}+\frac{2}{p}=1$ and $p \geq 4$, then automatically $u \in L^{\beta} B^0_{p,\infty}$ for $\beta=\frac{2p}{p-2}\leq p$. Thus Lemma \ref{lemma:classical} recovers the result of Shinbrot \cite{Shin} (see Figure \ref{fig:classical1}).
	\end{remark}

	\begin{figure*}[hb]
		\centering
		\begin{minipage}{0.48\textwidth}
			\centering
			\begin{tikzpicture}
			\draw [<->] (0,5) node[left]{$ \frac{1}{p}$}  --(0,0)  -- (5,0) node[below]{$ \frac{1}{\beta} $};
			\draw  (1.33, 0) -- node[below]{$ \frac{1}{3} $} (1.33, -0.1); 
			%\draw  (2, 0) -- node[below]{$ \frac{1}{2} $} (2, -0.1); 
			\draw  (4, 0) -- node[below]{$ 1 $} (4, -0.1); 
			\draw  (-0.1, 4) -- node[left]{$ 1 $} (0, 4);
			\draw  (-0.1, 2) -- node[left]{$ \frac{1}{2} $} (0, 2);
			\draw  (-0.1, 1.33) -- node[left]{$ \frac{1}{3} $} (0, 1.33); 
			\draw [dashed]  (0,0) --(1.33,1.33)  ;
			\draw [dashed] (0,2)  -- (1.33,1.33) ; \draw [dashed]   (1.33,1.33)  -- (4,0) ;   
			\draw [dashed] (1.33,0)  -- (1.33,1.33) ; \draw [dashed]  (1.33,1.33)  -- (1.33,4) ;   
			\draw [->] (4,1.5)node[right]{$\frac{2}{p}+\frac{1}{\beta} = 1$} -- (2.8,0.7)  ;
			\draw [dashed] (0,1.33) -- (1.33,1.33) ;
			\fill[pattern=north east lines,opacity=0.5] (1.33,1.33)-- (0,0) --(4,0)-- (1.33,1.33) ;
			
			\draw[->]    (2.7,2.5)node[above]{ $L^{\beta,w} B^{\frac{2}{\beta} + \frac{2}{p} -1 }_{p,\infty} $ } --(1.6,0.5) ;			 
			\end{tikzpicture}
			\caption{Regions for Theorem \ref{theorem:weak_lbg}}
			\label{fig:weak_lbg}
		\end{minipage}	
		\begin{minipage}{0.48\textwidth}
			\centering
			\begin{tikzpicture}
			\draw [<->] (0,5) node[left]{$ \frac{1}{p}$}  --(0,0)  -- (5,0) node[below]{$ \frac{1}{\beta} $};
			\draw  (1.33, 0) -- node[below]{$ \frac{1}{3} $} (1.33, -0.1); 
			%\draw  (2, 0) -- node[below]{$ \frac{1}{2} $} (2, -0.1); 
			\draw  (4, 0) -- node[below]{$ 1 $} (4, -0.1); 
			\draw  (-0.1, 4) -- node[left]{$ 1 $} (0, 4);
			\draw  (-0.1, 2) -- node[left]{$ \frac{1}{2} $} (0, 2);
			\draw  (-0.1, 1.33) -- node[left]{$ \frac{1}{3} $} (0, 1.33); 
			\draw [dashed]  (0,0) --(1.33,1.33)  ;
			\draw [dashed] (0,2)  -- (1.33,1.33) ; \draw    (1.33,1.33)  -- (4,0) ;   
			\draw [dashed] (1.33,0)  -- (1.33,1.33) ; \draw  (1.33,1.33)  -- (1.33,4) -- (5,4) ;   
			\draw [->] (2.5,-0.3)node[below]{$\frac{2}{p}+\frac{1}{\beta} = 1$} -- (2.7,0.6)  ;
			\draw [dashed] (0,1.33) -- (1.33,1.33) ;
			\fill[opacity=0.1] (1.33,4)--(1.33,1.33)-- (4,0) --(5,0)-- (5,4) ;
			
			\draw    (3.5,2)node[above]{ $L^{\beta} B^{\frac{5}{2\beta} + \frac{3}{p} -\frac{3}{2} }_{p,\infty} $ }   ;			 
			\end{tikzpicture}
			\caption{Regions for Theorem \ref{theorem:type2}}
			\label{fig:type2}
		\end{minipage}	
	\end{figure*}

	It is clear that Theorem \ref{theorem:weak_lbg} improves  classical results in the interior of the region where $\alpha = \frac{2}{\beta}+ \frac{2}{p}- 1 $ (See Figure \ref{fig:weak_lbg}). In particular, if
	$u \in L^{q,w} L^p$ with $\frac{2}{q}+\frac{2}{p}=1$ and $p > 4$, then $u \in L^{\beta} B^0_{p,\infty}$ for $\beta=\frac{2p}{p-2}< p$, and hence our condition \ref{eq:weak_lbg}
	is satisfied. Thus Theorem \ref{theorem:weak_lbg} extends the result of Shinbrot \cite{Shin} to weak-in-time Lebesgue spaces.
	In addition Figure \ref{fig:type2} shows that Theorem \ref*{theorem:type2} extends the condition \eqref{eq:previous_energy2} to the regime where $0 < \beta <1$. 
	
	It is worth noting that in \cite{LS1} the authors were able to obtain better scaling (better space regularity exponent) than Figure \ref{fig:classical1} in a small region where $\alpha = \frac{1}{\beta}+ \frac{3}{p}- 1$ for strong solutions up to the first time of blowup. However at the moment it seems that Figure \ref{fig:classical1} is optimal for general weak solutions in terms of space regularity exponent.

	\section{Preliminaries}\label{section3}
	
	\subsection{Notations}
	We denote by $A \lesssim B$ an estimate of the form $A \leq CB $ with some
	absolute constant $C$, and by $A \sim B$ an estimate of the form $C_1B \leq  A \leq  C_2B $ with some absolute constants $C_1$, $C_2$. We write $\| \cdot \|= \|\cdot \|_{L^p} $ for Lebesgue norms. The symbol $(\cdot,\cdot)$ stands for the $L^2$-inner product and $L^{\beta,w}$ stands for weak Lebesgue spaces. For any $p\in \mathbb{N}$ we let $\lambda_p=2^p$ be the standard dyadic number.
	
	\subsection{Littlewood-Paley decomposition}
	We briefly introduce a standard Littlewood-Paley decomposition. For a detailed background on harmonic analysis we refer to \cite{Ca}. Let $\chi : \mathbb{R}^+ \rightarrow \mathbb{R}$ be a smooth function so that $\chi(\xi) =1$ for $\xi \leq \frac{3}{4}$, and $\chi(\xi) =0$ for $\xi \geq 1$. We further define $\varphi(\xi)=\chi(\lambda_1^{-1 }\xi) -\varphi(\xi) $  and $\varphi_q(\xi) = \varphi(\lambda_q^{-1}\xi)$. For a tempered distribution vector field $u$ let us denote
	\begin{align*}
	u_q = \mathcal{F}^{-1}(\varphi_q)*u \quad \text{for } q>-1,  \qquad u_{-1} = u_q = \mathcal{F}^{-1}(\chi)*u,
	\end{align*}
	where $\mathcal{F}$ is the Fourier transform. With this we have $u=\sum_{q \geq -1}u_q$ in the sense of distribution.
	
	We use the following version of Bernstein's inequality.
	
	\begin{lemma}
		Let  $r \geq s \geq 1$. For any tempered distribution $u \in \mathcal{S}(\RR^3)$
		$$
		\|u_q \|_r \lesssim \lambda_q^{3(\frac{1}{s} - \frac{1}{r})} \|u_q \|_s
		$$
		holds for any $ -1 \leq q   \in \mathbb{Z} $, where the positive implicit constant is universal and independent of $q$.
	\end{lemma}
	
	Also let us finally note that the Besov space $B^{s}_{p,q}$ is the space consisting of all tempered distributions $u$ satisfying
	$$
	\| u\|_{B^{s}_{p,q}}:=   \big\| \lambda_r^s \|u_r \|_p \big\|_{l^q} < \infty.
	$$
	\subsection{Energy flux}
	We use the following truncated energy equality for 3D NSE through the wavenumber $\lambda_q$:
	\begin{equation}\label{eq:trun_energ}
	{\textstyle \frac{1}{2} }\|u_{\leq q}(t) \|_2^2 = {\textstyle \frac{1}{2}} \|u_{\leq q}(t_0) \|_2^2 + \int_{t_0}^t \left( -\nu \|\nabla u_{\leq q}(s) \|_2^2  + \Pi_{\leq q}(s) \right)\, ds,
	\end{equation}
	where $\Pi_{\leq q}$ is the energy flux through the wavenumber $\lambda_q$:
	\begin{equation}
	\Pi_{\leq q} = \int \tr(  (u \otimes u)_{\leq q} \cdot \nabla  u_{\leq q}  ) \, dx.
	\end{equation}

	The next result was proven in \cite{ccfs}, which we use for much of this paper.
	\begin{proposition}[Flux]
		For any vector field $u \in L^2$ we we have the following estimate for the energy flux:
		\begin{equation}\label{eq:fluxatq_est}
		|\Pi_{\leq q}| \lesssim \bigg[ \sum_{r < q} \lambda_r^\frac{2}{3} \|u_r \|_3^2  \lambda_{|r-q|}^{-\frac{4}{3}} \bigg]^\frac{3}{2} +\bigg[ \sum_{r \geq  q} \lambda_r^\frac{2}{3} \|u_r \|_3^2   \lambda_{|r-q|}^{-\frac{2}{3}} \bigg]^\frac{3}{2}.
		\end{equation}
	\end{proposition}

	\section{Proof of main results}\label{section4}
	\subsection{Energy equality for weak-in-time Onsager spaces}
	
	Thanks to \eqref{eq:trun_energ}, Theorem \ref{theorem:weak_lbg} is a direct consequence of the following.

	\begin{proposition} 
		Suppose a weak solution $u$ on $[0,T]$ satisfies 
		\begin{equation} \label{eq:p4.1assumption}
		u \in L^{\beta,w}(0,T;B^{\frac{2}{\beta}+\frac{2}{p}-1}_{p,\infty} ),
		\end{equation}
		for some $\frac{2}{p} + \frac{1}{\beta}<1$ and $p > \beta>0$. Then we have
		\begin{equation}\label{eq:flux-above}
		\limsup_{q \to \infty} \int_{0}^T |\Pi_{\leq q}(s)| \, ds =0.
		\end{equation}
	\end{proposition}
	\begin{proof}
		Throughout the proof we denote $f = \| u \|_{B^{\alpha}_{p,\infty}} $ and $\alpha =\frac{2}{\beta} + \frac{2}{p}-1 $. Let us also define 
		\[
		E_q= \{s\in [0,T]: f(s) \geq \lambda_q^\frac{2}{\beta}\}.
		\]
		It follows from \eqref{eq:p4.1assumption} that $|E_q| \lesssim \lambda_q^{-2}$.
		With this we split the energy flux as 
		$$
		\int_{0}^T |\Pi_{\leq q}(s)|\, ds \leq  \int_{E_q} |\Pi_{\leq q}(s)| \,ds+\int_{[0,T]\setminus E_q} |\Pi_{\leq q}(s)| \,ds.
		$$
		
		\textbf{Step 1:} Bounding $\int_{E_q} |\Pi_{\leq q}(s)| \,ds$.

		We first use the H\"older interpolation inequality to obtain
		$$
		\int_{E_q} |\Pi_{\leq q}(s)| \,ds  \lesssim \int_{E_q} \sum_{r} \lambda_r \|u_r \|_2^\frac{2p-6}{p-2} \| u_r \|_p^\frac{p}{p-2}\lambda_{|r-q|}^{-\frac{2}{3}} \, ds.
		$$
		It follows from the definition of Besov norms that
		\begin{equation}
		\int_{E_q} |\Pi_{\leq q}(s)|\, ds  \lesssim \int_{E_q} \sum_{r  } \|u_r \|_2^\frac{2p-6}{p-2} \lambda_r^{1-\frac{\alpha p}{p-2}}  \lambda_{|r-q|}^{-\frac{2}{3}}  f(s)^{\frac{ p}{p-2}}\, ds.
		\end{equation}

		Since $\frac{ p}{p-2} < \beta <p$, we can choose $\epsilon>0$ small enough so that
		\[
		\frac{p}{p-2} < \frac{\beta}{1+\epsilon}, \quad \epsilon_1:=  \frac{2}{p-2} \left(\frac{p}{\beta} -1\right) + 2\epsilon' > 0,
		\quad \text{and} \quad \epsilon_2: = \frac{2}{p-2} \left(\frac{p}{\beta} -1\right) - 2\epsilon' > 0.
		\]
		where $\epsilon'= \frac{ p}{\beta(p-2)}\epsilon$. Now we can use H\"older's inequality to raise the power of $f$.
		\begin{align*}
		\int_{E_q} |\Pi_{\leq q}(s)| \,ds  \lesssim & \sum_{r   }\lambda_{|r-q|}^{-\frac{2}{3} }  \bigg[\int_{E_q}  \lambda_r^{2}   \|u_r \|_2^2 \,ds \bigg]^{1- \frac{  p}{\beta (p-2)} -\epsilon' }  \\
		& \cdot \lambda_r^{2\epsilon' }\bigg[\int_{E_q} f^{\frac{\beta}{1+\epsilon}} \,  ds \bigg]^{\frac{ p}{\beta(p-2)}+\epsilon'} \cdot \sup_t \|u(t) \|_2^{\epsilon_1},
		\end{align*}
		where we note that $1- \frac{  p}{\beta (p-2)} -\epsilon' > 0$ thanks to the bound $\frac{p}{p-2} < \frac{\beta}{1+\epsilon} $.

		Due to \eqref{eq:p4.1assumption}, we have the following bound on the distribution function:
		\[
		\lambda_f(t)=|\{s:|f(s)|>t\}| \lesssim t^{-\beta}.
		\]
		Hence we obtain
		\begin{align*}
		\int_{E_q} f^{\frac{\beta}{1+\epsilon}} \,  ds = \frac{\beta}{1+\epsilon}\int_{\lambda_q^\frac{2}{\beta}}^\infty t^{\frac{\beta}{1+\epsilon} -1} \lambda_f(t) \,dt \lesssim  \int_{\lambda_q^\frac{2}{\beta}}^\infty t^{\frac{ - \epsilon \beta}{1+\epsilon} -1}  \, dt \lesssim \lambda_q^{\frac{ - 2\epsilon  }{1+\epsilon}  } .
		\end{align*}
		Since $\frac{p}{\beta(p-2)}+\epsilon' = \frac{\epsilon'(1+\epsilon)}{\epsilon}$, then we have
		\[
		\lambda_r^{2\epsilon' }\bigg[\int_{E_q} f^{\frac{\beta}{1+\epsilon}} \,  ds \bigg]^{\frac{ p}{\beta(p-2)}+\epsilon'} \lesssim1.
		\]
		
		Using this bound and the fact that the energy is bounded we arrive at
		$$
		\int_{E_q} |\Pi_{\leq q}(s)| \,ds \lesssim   \sum_{r   }\lambda_{|r-q|}^{-\frac{2}{3} }  \bigg[\int_{E_q}  \lambda_r^{2}   \|u_r(s) \|_2^2 \,ds \bigg]^{1- \frac{  p}{\beta (p-2)} -\epsilon'} \to 0,
		$$
		as $q\to 0$ due to the fact that
		\[
		\int_0^T \|\nabla u(s)\|_2^2\, ds < \infty.
		\]
		%Since $1- \frac{  p}{\beta (p-2)} -\epsilon'  >0$ due to $\epsilon,\epsilon' <<1$ by Jensen inequality we obtain $\int_{E_q} |\Pi_{\leq q}(s)| ds\to 0$ as $q \to \infty$.
		
		\textbf{Step 2:} Bounding $\int_{[0,T]\setminus E_q} |\Pi_{\leq q}(s)| ds$.
		
		Similarly to Step 1 we have 
		\begin{align*}
		\int_{ [0,T]\setminus E_q }|\Pi_{\leq q}(s)| ds  & \lesssim \int_{[0,T]\setminus E_q} \sum_{r  } \lambda_r \|u_r \|_2^\frac{2p-6}{p-2} \| u_r \|^\frac{p}{p-2}\lambda_{|r-q|}^{-\frac{2}{3}} \, ds\\
		&\lesssim \int_{[0,T]\setminus E_q} \sum_{r  } \lambda_r^{1-\frac{\alpha p}{p-2}} \|u_r \|_2^\frac{2p-6}{p-2} \lambda_{|r-q|}^{-\frac{2}{3}}  f(s)^\frac{p}{p-2} \,ds .
		\end{align*}
		
		H\"older's inequality in time gives
		\begin{align*}
		\int_{ [0,T]\setminus E_q }|\Pi_{\leq q}(s)|\, ds  \lesssim  & \sum_{r  }\lambda_{|r-q|}^{-\frac{2}{3}}  \bigg[   \int_{[0,T]\setminus E_q} \lambda_r^2\|u_r \|_2^{2 } \bigg]^{1- \frac{  p}{\beta (p-2)}  +\epsilon'} \\
		& \cdot \lambda_r^{-2\epsilon'} \bigg[ \int_{[0,T]\setminus E_q} f^{ \frac{\beta}{1-\epsilon}}\, ds \bigg]^{\frac{ p}{\beta(p-2) } -\epsilon'} \cdot \sup_t \|u(t) \|_2^{\epsilon_2},
		\end{align*}
		where $\epsilon'$ and $\epsilon_2$ are as above. Using the distribution function again, we obtain
		\begin{align*}
		\int_{[0,T]\setminus E_q} f^{\frac{\beta}{1-\epsilon}}  \, ds = \frac{\beta}{1+\epsilon}\int^{\lambda_q^\frac{2}{\beta}}_0  t^{\frac{\beta}{1-\epsilon} -1} \lambda_f(t)\, dt \lesssim  \int^{\lambda_q^\frac{2}{\beta}}_0 t^{\frac{  \epsilon \beta}{1-\epsilon} -1} \,  dt \lesssim \lambda_q^{\frac{  2\epsilon  }{1-\epsilon}  }.
		\end{align*}
		Since $\frac{p}{\beta(p-2)}-\epsilon' = \frac{\epsilon'(1-\epsilon)}{\epsilon}$, we have
		\[
		\lambda_r^{-2\epsilon' }\bigg[\int_{E_q} f^{\frac{\beta}{1-\epsilon}} \,  ds \bigg]^{\frac{ p}{\beta(p-2)}-\epsilon'} \lesssim1.
		\]
		
		Thus, reasoning as before,
		$$
		\int_{ [0,T]\setminus E_q }|\Pi_{\leq q}(s)|\, ds   \lesssim \sum_{r  } \lambda_{|r-q|}^{-\frac{2}{3}}  \bigg[   \int_{[0,T]\setminus E_q} \lambda_r^2 \|u_r(s) \|_2^{2 } \, ds\bigg]^{1- \frac{  p}{\beta (p-2)}  +\epsilon'} \to 0,
		$$
		as $q \to \infty$.
	\end{proof}

	\subsection{Extension in the region \texorpdfstring{$\frac{2}{p} + \frac{1}{\beta} \geq  1$ and $0<\beta\leq 3$ }{2/p + 1/beta <=  1}}
	
	In this case the proof is much simpler. We can use Sobolev or Bernstein's inequalities since the intermittency dimension $d$ is expected to be $0$. 
	\begin{proof}[Proof of Theorem \ref{theorem:type2}]
		Let us consider two sub-cases: $p\geq 3$ and $ 1 \leq p < 3$.
		
		First of all, when $p\geq 3$, by H\"older's inequality we obtain:
		$$
		\|u_r \|_3^3 \leq \|u_r \|_2^\frac{2p-6}{p-2}   \|u_r \|_p^\frac{p}{p-2}  .
		$$
		Since $\frac{p}{p-2} \geq \beta $, thanks to Bernstein's inequality we have
		$$
		\|u_r \|_3^3 \lesssim \|u_r(s) \|_2^{3-\beta}   \ \|u_r(s) \|_p^\beta \lambda_r^{\frac{3}{2} + \frac{3\beta}{p} -\frac{3\beta}{2} }.
		$$
		
		In the second case, when $ 1 \leq p < 3$, direct application of Bernstein's inequality also amounts to 
		$$
		\|u_r \|_3^3 \lesssim \|u_r(s) \|_2^{3-\beta}     \|u_r(s) \|_p^\beta \lambda_r^{\frac{3}{2} + \frac{3\beta}{p} -\frac{3\beta}{2} }.
		$$
		Note that in both cases the power of $\lambda_r$ is $\frac{3}{2} + \frac{3\beta}{p} -\frac{3\beta}{2} \geq 0$ due to the fact that $\frac{2}{p} + \frac{1}{\beta} \geq 1$.
		Therefore for any  $  p \geq 1$ we can proceed as
		\begin{equation}
		\int_{0}^{T}|\Pi_{\leq q}(s)| ds \leq  \int_{0}^{T^*}    \sum_{r  } \|u_r(s) \|_2^{3-\beta}   \ \|u_r(s) \|_p^\beta \lambda_r^{\frac{5}{2} + \frac{3\beta}{p} -\frac{3\beta}{2} }     \lambda_{|r-q|}^{-\frac{2}{3}}  ds.
		\end{equation}
		Since $0 <\beta \leq 3$ and the energy is bounded, the Dominated Convergence Theorem implies that
		$$
		\int_{0}^{T}|\Pi_{\leq q}(s)| ds  \lesssim \int_{0}^{T^*}    \sum_{r  }    \ \|u_r(s) \|_p^\beta \lambda_r^{\frac{5}{2} + \frac{3\beta}{p} -\frac{3\beta}{2} }     \lambda_{|r-q|}^{-\frac{2}{3}}  ds\to 0 \qquad \text{as  } q \to \infty.
		$$
		Therefore energy equality holds under condition \eqref{eq:type2}. 
		
	\end{proof}

	{}

\end{document}